\documentclass{imanum}
\usepackage{graphicx}
\usepackage{fancyhdr}
\usepackage{amsmath}
\usepackage{mathrsfs}
\usepackage{amsfonts}

\begin{document}

\title{Mean-square Convergence of a Symplectic Local Discontinuous Galerkin Method Applied to Stochastic Linear Schr\"{o}dinger Equation\thanks{This work was
supported by NNSFC (NO. 91130003,
NO. 11021101, NO. 11290142 and NO. 11471310).}}
\shorttitle{SYMPLECTIC LDG METHOD FOR SSE}

\author{%
{\sc
Chuchu Chen\thanks{Email: chenchuchu@lsec.cc.ac.cn},
Jialin Hong\thanks{Email: hjl@lsec.cc.ac.cn}
} \\[2pt]
State Key Laboratory of Scientific and Engineering Computing,
Institute of Computational Mathematics
and Scientific/Engineering Computing,
Academy of Mathematics and Systems Science,
Chinese Academy of Sciences,
Beijing, 100190, China\\[6pt]
{\sc and}\\[6pt]
{\sc Lihai Ji}\thanks{Corresponding author. Email: jilihai@lsec.cc.ac.cn}\\[2pt]
Institute of Applied Physics and Computational Mathematics, Beijing, 100094, China.
}
\shortauthorlist{CHUCHU CHEN, JIALIN HONG AND LIHAI JI}

\maketitle

\begin{abstract}
{In this paper, we investigate the mean-square convergence of a novel symplectic local discontinuous Galerkin method in ${\mathbb L}^2$-norm for stochastic linear Schr\"{o}dinger equation with multiplicative noise. It is shown that the mean-square error is bounded not only by the temporal and spatial step-sizes, but also by their ratio. The mean-square convergence rate with respect to the computational cost is derived under appropriate assumptions for initial data and noise.  Meanwhile, we show that the method preserves the discrete
charge conservation law which implies an ${\mathbb L}^{2}$-stability.}
{symplectic method; local discontinuous Galerkin method; stochastic linear Schr\"{o}dinger equation; ${\mathbb L}^{2}$-stability; charge conservation law; mean-square convergence.}
\end{abstract}

\section{Introduction}
\label{sec;introduction}
In this paper we consider the stochastic linear Schr\"odinger equation with multiplicative noise
\begin{equation}\label{NLS}
idu-(\Delta u+Q(x)u)dt=u\circ dW,\qquad u(x,0)=u_{0}(x),
\end{equation}
where $t\in[0,T]$, $x\in {\mathcal O}\subset {\mathbb R}^{d}$, $Q\in{\mathbb H}^{3}({\mathcal O})$ and periodic boundary condition holds.
Here, the $\circ$ in the last term in (\ref{NLS}) means that the product is of
Stratonovich type and $W$ on ${\mathbb L}^2({\mathcal O}^{d})$ is a
real-valued Wiener process with a filtered probability space $(\Omega,\mathcal{F},P,\{{\mathcal F}_{t}\}_{t\in[0,T]})$.
  It has the expansion form  $W(t,x,\omega)=\sum\limits_{k=0}^{\infty}\beta_{k}(t,\omega)\phi e_{k}(x),$ with $(e_{k})_{k\in \mathbb{N}^{d}}$ being an orthonormal basis of ${\mathbb L}^{2}({\mathcal O}^{d})$,
 $\{\beta_{k}\}_{k\in \mathbb{N}^d}$ being a sequence of independent Brownian motions and $\phi\in{\mathcal L}_{2}({\mathbb L}^2({\mathcal O}^d);{\mathbb H}^{\gamma}({\mathcal O}^d))$ being a Hilbert-Schmidt operator.
The phase flow of equation \eqref{NLS} is stochastic symplectic \citep[see][]{RefChu}, i.e., $$\bar{\omega}(t)=\int_{\mathcal O}\textrm{d}(r(t))\wedge \textrm{d}(s(t))dx=\bar{\omega}(0),$$
 with $r$ and $s$ being the real and imaginary parts of $u$, respectively, and its solution preserves the charge conservation law almost surely \citep[see][]{RefBouard3}, i.e.,
 $$\int_{\mathcal O}|u(x,t)|^{2}dx=\int_{\mathcal O}|u_{0}(x)|^{2}dx.$$

We propose a symplectic local discontinuous Galerkin method to discrete equation \eqref{NLS} in order to on one hand preserve the properties of the original problems as much as possible and on another hand combine the attractive properties of local discontinuous Galerkin method \citep[see, e.g.,][]{RefCockburn4,RefCockburn2,RefCockburn3}.
We refer interested readers to  \citep{RefOhannes} and references therein for the numerical simulation of the deterministic Schr\"{o}dinger equation based on local discontinuous Galerkin method.
Due to the reason that equation \eqref{NLS} is meaningful in the sense of integral,
 we apply the midpoint scheme to discretize the temporal direction at first avoiding dealing with double temporal-spatial integrals which is introduced by stochastic integral and local discontinuous Galerkin discretization. It is shown that the midpoint semi-discretization not only is a symplectic method,
 but also possesses the discrete charge conservation law.
 Furthermore, we show that the semi-discretization is of order $1$ in mean-square convergence sense via a direct approach while authors in \citep{RefChu} proved the same result via a fundamental convergence theorem on the mean-square convergence for the temporal semi-discretizations.
  The main difficulty lies in the analysis of the mean-square convergence order for the spatial direction, where
 we use local discontinuous Galerkin method to discrete the semi-discretized equation and obtain the full-discrete method which is called symplectic local discontinuous Galerkin method in this paper. We solve it by means of the standard approximation theory of projection operator, It\^{o} isometry and the adapted properties of processes $u$ and $W$.
 As a result we analysis the mean-square convergence error for the symplectic local discontinuous Galerkin method, and derive the mean-square convergence rate with respect to the computational cost under appropriate hypothesis on initial data and noise. Moreover theoretical analysis shows that the obtained full-discrete method is ${\mathbb L}^{2}$-stable and preserves the discrete charge conservation law.

The rest of this paper is organized as follows. In section 2, we propose the symplectic local
discontinuous Galerkin method for stochastic Schr\"odinger equation and derive the discrete charge conservation law. In section 3, we study the mean-square convergence of the obtained method and present the mean-square error estimation. Some proofs and calculations are postponed to the final appendices.
\section{The symplectic local discontinuous Galerkin method}
In this section, we will apply implicit midpoint scheme to \eqref{NLS} in the temporal direction, then we discretize the spatial direction by local discontinuous Galerkin method and obtain the full-discrete method.
\subsection{Temporal semi-discrete scheme}\label{3.1}
The midpoint scheme for \eqref{NLS} reads
\begin{equation}\label{Com}
iu^{n+1}=iu^{n}-\Delta t\Big(\Delta u^{n+\frac12}+Q(x)u^{n+\frac12}\Big)+u^{n+\frac12}\Delta \tilde{W}_{n},\; n=0,1,\cdots,N
\end{equation}
where $\Delta t$ is the time step size, $N=\frac{T}{\Delta t}$, $u^{n+\frac{1}{2}}=\frac{1}{2}(u^{n+1}+u^{n})$, and $\Delta \tilde{W}_{n}=\sum\limits_{k=0}^{\infty}\sqrt{\Delta t}\zeta_{k,n}^{\kappa}\phi e_{k}(x) $ with $\zeta_{k,n}^{\kappa}$ being the truncation of a $\mathcal{N}(0,1)$-distribution random variable
$\xi_{k,n}$:
\[
  \zeta_{k,n}^{\kappa}=\begin{cases} \kappa &\text{ if }\xi_{k,n}>\kappa;\\
                                             \xi_{n} &\text{ if }|\xi_{k,n}|\leq \kappa;\\
                                             -\kappa &\text{ if }\xi_{k,n}<-\kappa  \end{cases}
\]
with $\kappa:=\sqrt{4|\ln(\Delta t)|}$.
This choice is motivated by
the fact that standard Gaussian random variables  are unbounded for arbitrary
values of $\Delta t$,  \citep[see][]{RefMils} for more details. For the truncated Wiener process, we have the following properties:
\begin{equation}\label{trunW}
 \begin{split}
  &{\rm (i)}\qquad E\|\Delta \tilde{W}_{n}-\Delta W_{n}\|_{{\mathbb H}^1}^2\leq K\Delta t^3,\\
  &{\rm (ii)}\qquad E\|(\Delta \tilde{W}_{n})^2-(\Delta W_{n})^2\|_{{\mathbb H}^1}^2\leq K\Delta t^4,
  \end{split}
\end{equation}
where the constant $K$ depends on $\|\phi\|_{\mathcal{L}_{2}({\mathbb L}^2,{\mathbb H}^1)}$.
Based on the fact that $\tilde{W}$ is real-valued, by multiplying both sides of equation (\ref{Com}) by $u^{\star n+\frac{1}{2}}$, which is the conjugate of $u^{n+\frac{1}{2}}$, and then taking the imaginary part and integrating it over the whole space domain. We can get the discrete charge conservation law as follows.
\begin{proposition}
Under the periodic boundary conditions, the semi-discrete scheme \eqref{time} of the system \eqref{NLS} has the discrete charge conservation law, i.e.,
\begin{equation}
\int_{\mathcal O}|u^{n+1}(x)|^{2}dx=\int_{\mathcal O}|u^{n}(x)|^{2}dx,~~n=0,1,...,N.
\end{equation}
\end{proposition}

Furthermore, the semi-discrete scheme (\ref{time}) preserves the stochastic symplectic structure; \citep[see][]{RefChu}.

\begin{proposition}
The implicit midpoint scheme \eqref{time} for the system \eqref{NLS} is stochastic symplectic.
\end{proposition}

\subsection{Temporal-spatial full-discrete method}
In this subsection, we consider the local discontinuous Galerkin method for the system (\ref{time}) in the spatial direction and obtain the full-discrete method. To this end, we introduce some spatial-gird notations for the case $d=1$, ${\mathcal O}=[L_f,\,L_{r}]$ for simplicity, and the results hold for the general dimensional problems.
We denote the mesh by $I_{j}=[x_{j-\frac{1}{2}},x_{j+\frac{1}{2}}]$, for $1\leq j\leq J$, where
$
L_{f}=x_{\frac{1}{2}}<x_{\frac{3}{2}}<\cdot\cdot\cdot<x_{N+\frac{1}{2}}=L_{r}.
$
Let
$
\Delta x_{j}=x_{j+\frac{1}{2}}-x_{j-\frac{1}{2}},~1\leq j\leq J
$
with $h=\underset{1\leq j\leq J}{\max}\Delta x_{j}$ being the maximum mesh size.
Assume the mesh is regular, namely there is a constant $c>0$ independent of $h$ such that
$
\Delta x_{j}\geq ch,~~~1\leq j\leq J.
$

By decomposing the complex function $u^{n}$ into its real and imaginary parts
$u^{n}=r^{n}+is^{n}$
with $r^n$ and $s^n$ being real functions, we obtain the following first order semi-discrete system
\begin{equation}\label{time}
\begin{split}
r^{n+1}&=r^{n}+\Big((p_{x})^{n+\frac{1}{2}}+Q(x)s^{n+\frac{1}{2}} \Big)\Delta t+s^{n+\frac{1}{2}}\Delta \tilde{W}_{n},\\
p^{n+\frac{1}{2}}&=(s_{x})^{n+\frac{1}{2}},\\
s^{n+1}&=s^{n}-\Big((q_{x})^{n+\frac{1}{2}}+Q(x) r^{n+\frac{1}{2}}\Big)\Delta t-r^{n+\frac{1}{2}}\Delta \tilde{W}_{n},\\
q^{n+\frac{1}{2}}&=(r_{x})^{n+\frac{1}{2}}.
\end{split}
\end{equation}

We consider the local discontinuous Galerkin method for the system (\ref{time}) in the spatial direction and obtain the full-discrete method: find $r_{h},p_{h},s_{h},q_{h}\in V_{h}^{k}$, which now denotes real piecewise polynomial of degree at most $k$, such that, for all test functions $\nu_{h},\omega_{h},\alpha_{h},\beta_{h}\in V_{h}^{k}=\{\nu:\nu\in P^{k}(I_{j});~1\leq j\leq J\}$ with
$P^{k}(I_{j})$ being the set of polynomials of degree up to $k$ defined on the cell $I_{j}$.
\begin{equation}\label{LDG22}
\begin{split}
&\int_{I_{j}}r_{h}^{n+1}\nu_{h}dx-\int_{I_{j}}r_{h}^{n}\nu_{h}dx-\Delta t\Big[(\hat{p}^{n+\frac{1}{2}}\nu_{h}^{-})_{j+\frac{1}{2}}-(\hat{p}^{n+\frac{1}{2}}\nu_{h}^{+})_{j-\frac{1}{2}}\Big]\\
&~~~~~~+\Delta t\int_{I_{j}}\Big(p_{h}^{n+\frac{1}{2}}(\nu_{h})_{x}-s_{h}^{n+\frac{1}{2}}Q_{h}\nu_{h}\Big)dx-\int_{I_{j}}s_{h}^{n+\frac{1}{2}}\nu_{h}\Delta \tilde{W}_{n}dx=0,\\
&\int_{I_{j}}p_{h}^{n+\frac{1}{2}}\omega_{h}dx+\int_{I_{j}}s_{h}^{n+\frac{1}{2}}(\omega_{h})_{x}dx-\Big[(\hat{s}^{n+\frac{1}{2}}\omega_{h}^{-})_{j+\frac{1}{2}}-(\hat{s}^{n+\frac{1}{2}}\omega_{h}^{+})_{j-\frac{1}{2}}\Big]=0,\\
&\int_{I_{j}}s_{h}^{n+1}\alpha_{h}dx-\int_{I_{j}}s_{h}^{n}\alpha_{h}dx+\Delta t\Big[(\hat{q}^{n+\frac{1}{2}}\alpha_{h}^{-})_{j+\frac{1}{2}}-(\hat{q}^{n+\frac{1}{2}}\alpha_{h}^{+})_{j-\frac{1}{2}}\Big]\\
&~~~~~~-\Delta t\int_{I_{j}}\Big(q_{h}^{n+\frac{1}{2}}(\alpha_{h})_{x}-r_{h}^{n+\frac{1}{2}}Q_{h}\alpha_{h}\Big)dx+\int_{I_{j}}r_{h}^{n+\frac{1}{2}}\alpha_{h}\Delta \tilde{W}_{n}dx=0,\\
&\int_{I_{j}}q_{h}^{n+\frac{1}{2}}\beta_{h}dx+\int_{I_{j}}r_{h}^{n+\frac{1}{2}}(\beta_{h})_{x}dx-\Big[(\hat{r}^{n+\frac{1}{2}}\beta_{h}^{-})_{j+\frac{1}{2}}-(\hat{r}^{n+\frac{1}{2}}\beta_{h}^{+})_{j-\frac{1}{2}}\Big]=0.
\end{split}
\end{equation}
In the sequel, we denote by $(u_{h})^{+}_{j+\frac{1}{2}}$ and $(u_{h})^{-}_{j+\frac{1}{2}}$ the values of $u_{h}$ at $x_{j+\frac{1}{2}}$, from the right cell $I_{j+1}$, and from the left cell $I_{j}$, respectively.
And the numerical fluxes become
\begin{equation}\label{Fluxes}
\hat{p}=p^{+},~\hat{r}=r^{-},~\hat{q}=q^{+},~\hat{s}=s^{-},
\end{equation}
where we have omitted the half-integer indices $j+\frac{1}{2}$ as all quantities in (\ref{Fluxes}) are computed at the same points.

\begin{remark}
The choice for the fluxes (\ref{Fluxes}) is not unique. The important point is that $\hat{r}$ and $\hat{q}$, $\hat{s}$ and $\hat{p}$ should be chosen from different directions.
\end{remark}

With such a choice of fluxes (\ref{Fluxes}), we can get the first main result about discrete charge conservation law of the symplectic local discontinuous Galerkin method (\ref{LDG22}).
\begin{theorem}\label{Chagge}
Under the periodic boundary conditions, the symplectic local discontinuous Galerkin method \eqref{LDG22} has the discrete charge conservation law, i.e.,
\begin{equation}\label{Discrete-charge}
\int_{L_{f}}^{L_{r}}|u_{h}^{n+1}|^{2}dx=\int_{L_{f}}^{L_{r}}|u_{h}^{n}|^{2}dx, ~~n=0,1,2,...,N.
\end{equation}
\end{theorem}

\begin{proof}
To complete the proof of the discrete charge conservation law. First, we write (\ref{LDG22}) as the complex form. Denote $u_{h}^{n}=r_{h}^{n}+is_{h}^{n}, \psi_{h}^{n}=q_{h}^{n}+ip_{h}^{n}$, and take $\alpha_{h}=\nu_{h},\beta_{h}=\omega_{h}$, then (\ref{LDG22}) become
\begin{equation}\label{LDG1}
\begin{split}
&i\int_{I_{j}}u_{h}^{n+1}\nu_{h}dx-i\int_{I_{j}}u_{h}^{n}\nu_{h}dx-[(\hat{\psi}^{n+\frac{1}{2}}\nu_{h}^{-})_{j+\frac{1}{2}}-(\hat{\psi}^{n+\frac{1}{2}}\nu_{h}^{+})_{j-\frac{1}{2}}]\Delta t\\
&+\Delta t\int_{I_{j}}(\psi_{h}^{n+\frac{1}{2}}(\nu_{h})_{x}-u_{h}^{n+\frac{1}{2}}Q_{h}\nu_{h})dx-\int_{I_{j}}u_{h}^{n+\frac{1}{2}}\nu_{h}\Delta \tilde{W}_{n}dx=0,\\
&\int_{I_{j}}\psi_{h}^{n+\frac{1}{2}}\omega_{h}dx+\int_{I_{j}}u_{h}^{n+\frac{1}{2}}(\omega_{h})_{x}dx-\Big[(\hat{u}^{n+\frac{1}{2}}\omega_{h}^{-})_{j+\frac{1}{2}}-(\hat{u}^{n+\frac{1}{2}}\omega_{h}^{+})_{j-\frac{1}{2}}\Big]=0.
\end{split}
\end{equation}
where
\begin{equation}\label{flux}
\hat{u}=r_{h}^{-}+is_{h}^{-},~~\hat{\psi}=q_{h}^{+}+ip_{h}^{+}.
\end{equation}

Now, we take the complex conjugate for every terms in system (\ref{LDG1})
\begin{equation}\label{2.146}
\begin{split}
&-i\int_{I_{j}}(u_{h}^{\star})^{n+1}\nu_{h}^{\star}dx+i\int_{I_{j}}(u_{h}^{\star})^{n}\nu_{h}^{\star}dx-\Delta t\Big[(\hat{\psi}^{\star n+\frac{1}{2}}\nu_{h}^{\star-})_{j+\frac{1}{2}}-(\hat{\psi}^{\star n+\frac{1}{2}}\nu_{h}^{\star+})_{j-\frac{1}{2}}\Big]\\
&+\Delta t\int_{I_{j}}\Big(\psi_{h}^{\star n+\frac{1}{2}}(\nu_{h}^{\star})_{x}-u_{h}^{\star n+\frac{1}{2}}Q_{h}\nu_{h}^{\star}\Big)dx-\int_{I_{j}}u_{h}^{\star n+\frac{1}{2}}\nu_{h}^{\star}\Delta \tilde{W}_{n}dx=0,\\
&\int_{I_{j}}\psi_{h}^{\star n+\frac{1}{2}}\omega_{h}^{\star}dx+\int_{I_{j}}u_{h}^{\star n+\frac{1}{2}}(\omega_{h})_{x}dx-\Big[(\hat{u}^{\star n+\frac{1}{2}}\omega_{h}^{\star-})_{j+\frac{1}{2}}-(\hat{u}^{\star n+\frac{1}{2}}\omega_{h}^{\star+})_{j-\frac{1}{2}}\Big]=0.
\end{split}
\end{equation}

We introduce a short-hand notation
\begin{equation}
\begin{split}
\mathfrak{H}_{j}(u_{h}^{n},\psi_{h}^{n};\nu_{h},\omega_{h})&=i\int_{I_{j}}u_{h}^{n+1}\nu_{h}dx-i\int_{I_{j}}u_{h}^{n}\nu_{h}dx-\Delta t\int_{I_{j}}\psi_{h}^{n+\frac{1}{2}}\omega_{h}dx\\
&+\Delta t\int_{I_{j}}\Big(\psi_{h}^{n+\frac{1}{2}}(\nu_{h})_{x}-u_{h}^{n+\frac{1}{2}}Q_{h}\nu_{h}\Big)dx-\int_{I_{j}}u_{h}^{n+\frac{1}{2}}\nu_{h}\Delta \tilde{W}_{n}dx\\
&-\Delta t\int_{I_{j}}u_{h}^{n+\frac{1}{2}}(\omega_{h})_{x}dx-\Delta t\Big[(\hat{\psi}^{n+\frac{1}{2}}\nu_{h}^{-})_{j+\frac{1}{2}}-(\hat{\psi}^{n+\frac{1}{2}}\nu_{h}^{+})_{j-\frac{1}{2}}\Big]\\
&+\Delta t\Big[(\hat{u}^{n+\frac{1}{2}}\omega_{h}^{-})_{j+\frac{1}{2}}-(\hat{u}^{n+\frac{1}{2}}\omega_{h}^{+})_{j-\frac{1}{2}}\Big].
\end{split}
\end{equation}
Then from (\ref{2.146}), we also have the expression of $\mathfrak{H}_{j}^{\star}(u_{h}^{n},\psi_{h}^{n};\nu_{h},\omega_{h})$.
If we take $\nu_{h}=u_{h}^{\star n+\frac{1}{2}},~\omega_{h}=\psi_{h}^{\star n+\frac{1}{2}}$ in both functions $\mathfrak{H}_{j}(u_{h}^{n},\psi_{h}^{n};\nu_{h},\omega_{h})$ and $\mathfrak{H}_{j}^{\star}(u_{h}^{n},\psi_{h}^{n};\nu_{h},\omega_{h})$, both functions are zero. Hence we obtain
\begin{equation}\label{HH}
\mathfrak{H}_{j}(u_{h}^{n},\psi_{h}^{n};u_{h}^{\star n+\frac{1}{2}},\psi_{h}^{\star n+\frac{1}{2}})-\mathfrak{H}_{j}^{\star}(u_{h}^{n},p_{h}^{n};u_{h}^{\star n+\frac{1}{2}},\psi_{h}^{\star n+\frac{1}{2}})=0.
\end{equation}

With (\ref{flux}) of the numerical fluxes, then (\ref{HH}) becomes
\begin{equation}\label{19}
\begin{split}
i&\int_{I_{j}}\Big(|u_{h}^{n+1}|^{2}-|u_{h}^{n}|^{2}\Big)dx+\underbrace{\Delta t\int_{I_{j}}\Big(\psi_{h}^{n+\frac{1}{2}}(u_{h}^{\star n+\frac{1}{2}})_{x}+u_{h}^{\star n+\frac{1}{2}}(\psi_{h}^{n+\frac{1}{2}})_{x}\Big)dx}_{A}\\
&-\underbrace{\Delta t\int_{I_{j}}\Big(\psi_{h}^{\star n+\frac{1}{2}}(u_{h}^{n+\frac{1}{2}})_{x}+u_{h}^{n+\frac{1}{2}}(\psi_{h}^{\star n+\frac{1}{2}})_{x}\Big)dx}_{B}
-\underbrace{\Delta t\Big[(\psi_{h}^{n+\frac{1}{2}+}u_{h}^{\star n+\frac{1}{2}-})_{j+\frac{1}{2}}-(\psi_{h}^{\star n+\frac{1}{2}+}u_{h}^{n+\frac{1}{2}-})_{j+\frac{1}{2}}\Big]}_{C}\\
&+\underbrace{\Delta t\Big[(u_{h}^{n+\frac{1}{2}-}\psi_{h}^{\star n+\frac{1}{2}-})_{j+\frac{1}{2}}-(u_{h}^{\star n+\frac{1}{2}-}\psi_{h}^{n+\frac{1}{2}-})_{j+\frac{1}{2}}\Big]}_{D}
+\underbrace{\Delta t\Big[(\psi_{h}^{n+\frac{1}{2}+}u_{h}^{\star n+\frac{1}{2}+})_{j-\frac{1}{2}}-(\psi_{h}^{\star n+\frac{1}{2}+}u_{h}^{n+\frac{1}{2}+})_{j-\frac{1}{2}}\Big]}_{G}\\
&-\underbrace{\Delta t\Big[(u_{h}^{n+\frac{1}{2}-}\psi_{h}^{\star n+\frac{1}{2}+})_{j-\frac{1}{2}}-(u_{h}^{\star n+\frac{1}{2}-}\psi_{h}^{n+\frac{1}{2}+})_{j-\frac{1}{2}}\Big]}_{H}=0.
\end{split}
\end{equation}

By the chain of rule, we can derive
\begin{equation*}
\begin{split}
A=\Delta t\int_{I_{j}}(\psi_{h}^{n+\frac12}u_{h}^{\star n+\frac{1}{2}})_{x}dx=
\Delta t\Big[(\psi_{h}^{n+\frac{1}{2} -}u_{h}^{\star n+\frac{1}{2}-})_{j+\frac{1}{2}}-(\psi_{h}^{n+\frac{1}{2}+}u_{h}^{\star n+\frac{1}{2}+})_{j-\frac{1}{2}}\Big],\\
B=\Delta t\int_{I_{j}}(\psi_{h}^{\star n+\frac{1}{2}}u_{h}^{n+\frac12})_{x}dx=
\Delta t\Big[(u_{h}^{n+\frac{1}{2}-}\psi_{h}^{\star n+\frac{1}{2}-})_{j+\frac{1}{2}}-(u_{h}^{n+\frac{1}{2}+}\psi_{h}^{\star n+\frac{1}{2}+})_{j-\frac{1}{2}}\Big],
\end{split}
\end{equation*}
then
\begin{equation}\label{AB}
A-B=2i\Delta t\Big[\textrm{Im}(\psi_{h}^{n+\frac{1}{2}-}u_{h}^{\star n+\frac{1}{2}-})_{j+\frac{1}{2}}-\textrm{Im}(\psi_{h}^{n+\frac{1}{2}+}u_{h}^{\star n+\frac{1}{2}+})_{j-\frac{1}{2}}\Big].
\end{equation}

After some simple algebraic manipulation $a-a^{\star}=2i\textrm{Im}(a),~a\in \mathcal{C}$, we have
\begin{equation}\label{CD}
\begin{split}
C&=2i\Delta t\textrm{Im}(\psi_{h}^{n+\frac{1}{2}+}u_{h}^{\star n+\frac{1}{2}-})_{j+\frac{1}{2}},~~D=-2i\Delta t\textrm{Im}(\psi_{h}^{n+\frac{1}{2}-}u_{h}^{\star n+\frac{1}{2}-})_{j+\frac{1}{2}},\\
G&=2i\Delta t\textrm{Im}(\psi_{h}^{n+\frac{1}{2}+}u_{h}^{\star n+\frac{1}{2}+})_{j-\frac{1}{2}},~~H=-2i\Delta t\textrm{Im}(\psi_{h}^{n+\frac{1}{2}+}u_{h}^{\star n+\frac{1}{2}-})_{j-\frac{1}{2}}.
\end{split}
\end{equation}

We combine all these equalities (\ref{19}), (\ref{AB}) and (\ref{CD}) to obtain
\begin{equation*}
\int_{I_{j}}(|u_{h}^{n+1}|^{2}-|u_{h}|^{n})dx+\hat{\Phi}^{n+\frac{1}{2}}_{j+\frac{1}{2}}-\hat{\Phi}^{n+\frac{1}{2}}_{j-\frac{1}{2}}=0,
\end{equation*}
where the numerical entropy flux is given by
\begin{equation*}
\hat{\Phi}^{n+\frac{1}{2}}=-2\Delta t\textrm{Im}(\psi_{h}^{n+\frac{1}{2}+}u_{h}^{\star n+\frac{1}{2}-}).
\end{equation*}

Summing up over $j$, the flux terms vanish because of the periodic boundary conditions. Thus we finish the proof.
\end{proof}

\begin{corollary}
The discrete charge conservation law trivially implies an $L^{2}$-stability of the numerical solution.
\end{corollary}

\section{Error estimates for the full-discrete method}
In this section, we will state the error estimate of the symplectic local discontinuous Galerkin method for the problem \eqref{NLS} with $d=1$. In the sequel, $\mathbb{E}$ denotes an expectation operator of a random variable, and $K,C$ are  constants depending on $\|Q\|_{{\mathbb H}^{3}}$, the finial time $T$ and the norm of $u_0$, but independent of $h$ and $n$. They may change from line to line.

In order to obtain the error estimate to the symplectic local discontinuous Galerkin method (\ref{LDG22}) with the fluxes (\ref{Fluxes}), we divide the error into two parts:
\begin{equation}
\|u(t_{n})-u_{h}^{n}\|^{2}\leq \underbrace{\|u(\cdot,t_{n})-u^{n}\|^{2}}_{\textrm{Temporal~error}}+\underbrace{\|u^{n}-u_{h}^{n}\|^{2}}_{\textrm{Spatial~error}}.
\end{equation}

\subsection{Temporal error}
To obtain the temporal error estimate, we need some regularity results of the numerical solution $u^{n}(x)$ for \eqref{time}. We state it in the following two lemmas. The proof of these lemmas will be given in Appendix A and Appendix B, respectively.
\begin{lemma}\label{th4}
  Assume that  $Q\in{\mathbb H}^{\gamma}$ and $\mathbb{E}\|u^{0}\|_{{\mathbb H}^{\gamma}}^{2p}<\infty,\;\gamma=0,1,\cdots$ and $\phi\in\mathcal{L}_{2}({\mathbb L}^2;{\mathbb H}^{\gamma})$. We have the following regularity of temporal semi-discretization, i.e., for $p\geq 1$,
 \begin{equation}
 \mathbb{E}\|u^n\|_{{\mathbb H}^{\gamma}}^{2p}\leq K,~~n=1,2,...,N.
 \end{equation}
\end{lemma}

\begin{lemma}\label{th66}
  Given $\gamma=1,2,\cdots$, and assume  $Q\in{\mathbb H}^{\gamma}$, $u^{0}\in L^{2p}(\Omega;{\mathbb H}^{\gamma})$ and $\phi\in\mathcal{L}_{2}({\mathbb L}^2;{\mathbb H}^{\gamma})$, then we have holder continuity in temporal direction, i.e., for $p\geq 1$,
  \begin{equation*}
  \mathbb{E}\|u^{n+1}-u^{n}\|_{{\mathbb H}^{\gamma-1}}^{2p}\leq K\Delta t^{p},~~n=1,2,...,N.
  \end{equation*}
\end{lemma}

Now we are in a position to establish an error estimate of the semi-discrete method (\ref{time}) by virtue of these two lemmas.
\begin{theorem}\label{Time}
Assume that $u_{0}\in L^2(\Omega; {\mathbb H}^3)$, $Q\in{\mathbb H}^{3}$ and $\phi\in\mathcal{L}_{2}({\mathbb L}^2;{\mathbb H}^{3})$ then it is of the mean-square order 1, i.e.,
\begin{equation*}
\Big(\mathbb{E}\|u(t_{n})-u^n\|_{{\mathbb L}^2}^2\Big)^{1/2}\leq K\Delta t.
\end{equation*}
\end{theorem}
\begin{proof}
From \eqref{iterate} (see Appendix A) and \eqref{NLS}, it follows
\begin{align}\label{3}
   u^{n+1}=\hat{S}_{\Delta t}^{n+1}u^{0}-i\Delta t\sum_{\ell=1}^{n+1}\hat{S}_{\Delta t}^{n+1-\ell}T_{\Delta t}Qu^{\ell-\frac12}
    -i\sum_{\ell=1}^{n+1}\hat{S}_{\Delta t}^{n+1-\ell}T_{\Delta t}u^{\ell-\frac12}\Delta \tilde{W}_{\ell-1},
\end{align}
and
\begin{equation}\label{4}
\begin{split}
  u(t_{n+1})&=S(t_{n+1})u^{0}-i\int_{0}^{t_{n+1}}S(t_{n+1}-\tau)Qu(\tau)d\tau-i\int_{0}^{t_{n+1}}S(t_{n+1}-\tau)u(\tau)\circ dW(\tau)\\
  &=S(t_{n+1})u^0-i\sum_{\ell=1}^{n+1}\int_{t_{\ell-1}}^{t_{\ell}}S(t_{n+1}-\tau)Qu(\tau)d\tau-i\sum_{\ell=1}^{n+1}\int_{t_{\ell-1}}^{t_{\ell}}S(t_{n+1}-\tau)u(\tau) \circ dW(\tau).
  \end{split}
\end{equation}
Subtract \eqref{3} from \eqref{4} leads to
\begin{align*}
  u(t_{n+1})-u^{n+1}=&\underbrace{\Big(S(t_{n+1})-\hat{S}_{\Delta t}^{n+1}\Big)u^0}-i\sum_{\ell=1}^{n+1}\Big(\int_{t_{\ell-1}}^{t_{\ell}}S(t_{n+1}-\tau)Qu(\tau)d\tau-\Delta t\hat{S}_{\Delta t}^{n+1-\ell}T_{\Delta t}Qu^{\ell-\frac12}\Big)\\
  &-i\sum_{\ell=1}^{n+1}\Big(\int_{t_{\ell-1}}^{t_{\ell}}S(t_{n+1}-\tau)u(\tau) \circ dW(\tau)-\hat{S}_{\Delta t}^{n+1-\ell}T_{\Delta t}u^{\ell-\frac12}\Delta \tilde{W}_{\ell-1}\Big)\\
  =:&\mathcal{A}+\mathcal{B}+\mathcal{C}.
\end{align*}
We will estimate them separately.

$\bullet$ \textbf{The first term $\mathcal{A}$.}

From \citep{RefBouard4}, we know that $\|S(t_{n+1})-\hat{S}_{\Delta t}^{n+1}\|_{\mathcal{L}({\mathbb H}^3,{\mathbb L}^2)}\leq K\Delta t$. Thus, \[ \mathbb{E}\|\mathcal{A}\|_{{\mathbb L}^2}^2\leq K\mathbb{E}\|u^0\|_{{\mathbb H}^3}^2\Delta t^2\leq K\Delta t^2. \]

$\bullet$ \textbf{The second term $\mathcal{B}$.}

To estimate $\mathcal{B}$, we insert one term
\[ \pm i\sum_{\ell=1}^{n+1}\int_{t_{\ell-1}}^{t_{\ell}}S(t_{n+1}-r)Qu_{t_{\ell-1},u^{\ell-1}}(\tau)d\tau \]
into the expression of $\mathcal{B}$ and we have
\begin{align*}
  \mathcal{B}&=-i \sum_{\ell=1}^{n+1}\int_{t_{\ell-1}}^{t_{\ell}}S(t_{n+1}-\tau)Q\Big(u(\tau)-u_{t_{\ell-1},u^{\ell-1}}(\tau)\Big)d\tau
  \\
  &~~~~~~-i \sum_{\ell=1}^{n+1}\int_{t_{\ell-1}}^{t_{\ell}}\Big(S(t_{n+1}-r)Qu_{t_{\ell-1},u^{\ell-1}}(\tau)-\hat{S}_{\Delta t}^{n+1-\ell}T_{\Delta t}Qu^{\ell-\frac12}\Big)d\tau\\
  &=:\mathcal{B}^1+\mathcal{B}^2.
\end{align*}
To estimate term $\mathcal{B}^1$, we present the estimate of $u(\tau)-u_{t_{\ell-1},u^{\ell-1}}(\tau)$ by their expression,
\begin{align*}
  u(\tau)-u_{t_{\ell-1},u^{\ell-1}}(\tau)=&S(\tau-t_{\ell-1})(u(t_{\ell-1})-u^{\ell-1})\\
  &-i\int_{t_{\ell-1}}^{\tau}S(\tau-t_{\ell-1}-\rho)Q\big(u(\rho)-u_{t_{\ell-1},u^{\ell-1}}(\rho)\big)d\rho\\
  &-i\int_{t_{\ell-1}}^{\tau}S(\tau-t_{\ell-1}-\rho)\big(u(\rho)-u_{t_{\ell-1},u^{\ell-1}}(\rho)\big)\circ dW(\rho).
\end{align*}

Therefore, from Gronwall's inequality, we know $\mathbb{E}\|u(\tau)-u_{t_{\ell-1},u^{\ell-1}}(\tau)\|_{{\mathbb L}^2}^2\leq K\mathbb{E}\|u(t_{\ell-1})-u^{\ell-1}\|_{{\mathbb L}^2}^2$,
 and for term $B^1$
\[ \mathbb{E}\|\mathcal{B}^1\|_{{\mathbb L}^2}^2\leq K\Delta t\sum_{\ell=1}^{n+1}\|u(t_{\ell-1})-u^{\ell-1}\|_{{\mathbb L}^2}^2. \]
We split term $\mathcal{B}^2$ further as follows
\begin{align*}
 \mathcal{ B}^2=&-i \sum_{\ell=1}^{n+1}\int_{t_{\ell-1}}^{t_{\ell}}\Big(S(t_{n+1}-r)-\hat{S}_{\Delta t}^{n+1-\ell}T_{\Delta t}\Big)Qu_{t_{\ell-1},u^{\ell-1}}(\tau)d\tau
  \\
  &-i \sum_{\ell=1}^{n+1}\int_{t_{\ell-1}}^{t_{\ell}}\hat{S}_{\Delta t}^{n+1-\ell}T_{\Delta t}Q\Big(u_{t_{\ell-1},u^{\ell-1}}(\tau)-u^{\ell-1}\Big)d\tau\\
    &-i \Delta t\sum_{\ell=1}^{n+1}\hat{S}_{\Delta t}^{n+1-\ell}T_{\Delta t}Q\Big(u^{\ell}-u^{\ell-1}\Big)\\
    =:&\mathcal{B}_a^2+\mathcal{B}_b^2+\mathcal{B}_c^2.
\end{align*}
For term $B_a^2$, based on $\|S(t_{n})-\hat{S}_{\Delta t}^{n}\|_{\mathcal{L}({\mathbb H}^3;{\mathbb L}^2)}\leq K\Delta t$ , $\|I-T_{\Delta}\|_{\mathcal{L}({\mathbb H}^3;{\mathbb L}^2)}\leq K\Delta t$ and Lemma \ref{th4}, we have
\begin{equation*}
\mathbb{E}\|\mathcal{B}_a^2\|_{{\mathbb L}^2}^2\leq K\Delta t^2.
\end{equation*}
To estimate term $\mathcal{B}_b^2$, we insert the expression of $u_{t_{\ell-1},u^{\ell-1}}(\tau)-u^{\ell-1}$ into it and we have
\begin{align*}
  \mathcal{B}_b^2=&-i \sum_{\ell=1}^{n+1}\int_{t_{\ell-1}}^{t_{\ell}}\hat{S}_{\Delta t}^{n+1-\ell}T_{\Delta t}Q\Big[(S(\tau-t_{\ell-1})-I)u^{\ell-1}d\tau\\
  &-i\int_{t_{\ell-1}}^{\tau}S(\tau-\rho)(Q-\frac{i}{2})u_{t_{\ell-1},u^{\ell-1}}(\rho)d\rho\Big]d\tau\\
  &- \sum_{\ell=1}^{n+1}\int_{t_{\ell-1}}^{t_{\ell}}\hat{S}_{\Delta t}^{n+1-\ell}T_{\Delta t}Q\int_{t_{\ell-1}}^{\tau}S(\tau-\rho)u_{t_{\ell-1},u^{\ell-1}}(\rho)dW(\rho)d\tau.
\end{align*}
The estimate of the first term is similar to before and is bounded by $K\Delta t^2$.

Concerning the second term, we employ Fubini's theorem and It\^o isometry and Lemma \ref{th4},
\begin{align*}
  &\mathbb{E}\Big\|- \sum_{\ell=1}^{n+1}\int_{t_{\ell-1}}^{t_{\ell}}\hat{S}_{\Delta t}^{n+1-\ell}T_{\Delta t}Q\int_{t_{\ell-1}}^{\tau}S(\tau-\rho)u_{t_{\ell-1},u^{\ell-1}}(\rho)dW(\rho)d\tau\Big\|_{{\mathbb L}^2}^2\\
 & =\mathbb{E}\Big\|- \sum_{\ell=1}^{n+1}\int_{t_{\ell-1}}^{t_{\ell}}\hat{S}_{\Delta t}^{n+1-\ell}T_{\Delta t}Q\int_{\rho}^{t_{\ell}}S(\tau-\rho)u_{t_{\ell-1},u^{\ell-1}}(\rho)d\tau dW(\rho)\Big\|_{{\mathbb L}^2}^2\\
 & \leq K\Delta t^2.
\end{align*}
The estimate of term $\mathcal{B}_c^2$ is similar to that of term $\mathcal{B}_b^2$, by replacing the expression of $u^{\ell}-u^{\ell-1}$. Combining all the above inequalities, we obtain
the desired estimate of $\mathcal{B}$
\[ \mathbb{E}\|\mathcal{B}\|_{{\mathbb L}^2}^2\leq K\Delta t^2+ K\Delta t\sum_{\ell=1}^{n+1}\|u(t_{\ell-1})-u^{\ell-1}\|_{{\mathbb L}^2}^2.\]

$\bullet$ \textbf{The third term $\mathcal{C}$.}

 To estimate $\mathcal{C}$, we change Stratonovich integral into It\^o one with noting that $F_{\phi}=\sum_{\ell\in{\mathbb N}^{d}}\big(\phi e_{\ell}(x)\big)^2$,
 \begin{align*}
  \mathcal{ C}=&-\frac12\sum_{\ell=1}^{n+1}\int_{t_{\ell-1}}^{t_{\ell}}S(t_{n+1}-\tau)u(\tau)F_{\phi}d\tau\\
  & -i\sum_{\ell=1}^{n+1}\int_{t_{\ell-1}}^{t_{\ell}}S(t_{n+1}-\tau)u(\tau)dW(\tau)
   +i\sum_{\ell=1}^{n+1}\hat{S}_{\Delta t}^{n+1-\ell}T_{\Delta t}u^{\ell-\frac12}\Delta \tilde{W}_{\ell-1}.
 \end{align*}
 We split it further
 \begin{small}
 \begin{align*}
   \mathcal{C}=&-i\sum_{\ell=1}^{n+1}\int_{t_{\ell-1}}^{t_{\ell}}S(t_{n+1}-\tau)\Big(u(\tau)-u_{t_{\ell-1},u^{\ell-1}}(\tau)\Big)dW(\tau)
   -i\sum_{\ell=1}^{n+1}\int_{t_{\ell-1}}^{t_{\ell}}\Big(S(t_{n+1}-\tau)-\hat{S}_{\Delta t}^{n+1-\ell}T_{\Delta t}\Big)u_{t_{\ell-1},u^{\ell-1}}(\tau)dW(\tau)\\
   &
   -i\sum_{\ell=1}^{n+1}\int_{t_{\ell-1}}^{t_{\ell}}\hat{S}_{\Delta t}^{n+1-\ell}T_{\Delta t}\Big(u_{t_{\ell-1},u^{\ell-1}}(\tau)-u^{\ell-1}\Big)dW(\tau)
   +\frac{i}{2}\sum_{\ell=1}^{n+1}\hat{S}_{\Delta t}^{n+1-\ell}T_{\Delta t}\Big(u^{\ell}-u^{\ell-1}\Big)\Delta \tilde{W}_{\ell-1}\\
   &-\frac12\sum_{\ell=1}^{n+1}\int_{t_{\ell-1}}^{t_{\ell}}S(t_{n+1}-\tau)u(\tau)F_{\phi}d\tau+i\sum_{\ell=1}^{n+1}\hat{S}_{\Delta t}^{n+1-\ell}T_{\Delta t}u^{\ell-\frac12}\Big(\Delta \tilde{W}_{\ell-1}-\Delta W_{\ell-1}\Big).
 \end{align*}
 \end{small}

 By replacing the expressions of $u_{t_{\ell-1},u^{\ell-1}}(\tau)-u^{\ell-1}$ and $u^{\ell}-u^{\ell-1}$ into the above equation, we have
 \begin{small}
 \begin{equation}\label{6}
 \begin{split}
   \mathcal{C}=&-i\sum_{\ell=1}^{n+1}\int_{t_{\ell-1}}^{t_{\ell}}S(t_{n+1}-\tau)\Big(u(\tau)-u_{t_{\ell-1},u^{\ell-1}}(\tau)\Big)dW(\tau)
   -i\sum_{\ell=1}^{n+1}\int_{t_{\ell-1}}^{t_{\ell}}\Big(S(t_{n+1}-\tau)-\hat{S}_{\Delta t}^{n+1-\ell}T_{\Delta t}\Big)u_{t_{\ell-1},u^{\ell-1}}(\tau)dW(\tau)\\
  & -i\sum_{\ell=1}^{n+1}\int_{t_{\ell-1}}^{t_{\ell}}\hat{S}_{\Delta t}^{n+1-\ell}T_{\Delta t}\Big((S(\tau-t_{\ell-1})-I)u^{\ell-1}-i\int_{t_{\ell-1}}^{\tau}S(\tau-\rho)(Q-\frac{i}{2})u_{t_{\ell-1},u^{\ell-1}}(\rho)d\rho\Big)dW(\tau)\\
  &+\frac{i}{2}\sum_{\ell=1}^{n+1}\hat{S}_{\Delta t}^{n+1-\ell}T_{\Delta t}\Big((\hat{S}_{\Delta t}-I)u^{\ell-1}-i\Delta tT_{\Delta t}Qu^{\ell-\frac12}\Big)\Delta \tilde{W}_{\ell-1}
  -\frac12\sum_{\ell=1}^{n+1}\int_{t_{\ell-1}}^{t_{\ell}}S(t_{n+1}-\tau)\Big(u(\tau)-u_{t_{\ell-1},u^{\ell-1}}(\rho)\Big)F_{\phi}d\tau\\
  &-\sum_{\ell=1}^{n+1}\int_{t_{\ell-1}}^{t_{\ell}}\hat{S}_{\Delta t}^{n+1-\ell}T_{\Delta t}\int_{t_{\ell-1}}^{\tau}S(\tau-\rho)u_{t_{\ell-1},u^{\ell-1}}(\rho)dW(\rho)dW(\tau)
  +\frac12\sum_{\ell=1}^{n+1}\hat{S}_{\Delta t}^{n+1-\ell}T_{\Delta t}^2u^{\ell-\frac12}(\Delta \tilde{W}_{\ell-1})^2\\
  &-\frac12\sum_{\ell=1}^{n+1}\int_{t_{\ell-1}}^{t_{\ell}}S(t_{n+1}-\tau)u_{t_{\ell-1},u^{\ell-1}}(\rho)F_{\phi}d\tau+i\sum_{\ell=1}^{n+1}\hat{S}_{\Delta t}^{n+1-\ell}T_{\Delta t}u^{\ell-\frac12}\Big(\Delta \tilde{W}_{\ell-1}-\Delta W_{\ell-1}\Big).
 \end{split}
 \end{equation}
 \end{small}

 We pay more attention to the last three lines, denoted by $\mathcal{D}$, because other terms can be estimated as before, and are bounded by
 $K\Delta t^2+ K\Delta t\sum\limits_{\ell=1}^{n+1}\|u(t_{\ell-1})-u^{\ell-1}\|_{{\mathbb L}^2}^2$.

We have
 \begin{align*}
   &-\sum_{\ell=1}^{n+1}\int_{t_{\ell-1}}^{t_{\ell}}\hat{S}_{\Delta t}^{n+1-\ell}T_{\Delta t}\int_{t_{\ell-1}}^{\tau}S(\tau-\rho)u_{t_{\ell-1},u^{\ell-1}}(\rho)dW(\rho)dW(\tau)\\
   &=-\sum_{\ell=1}^{n+1}\int_{t_{\ell-1}}^{t_{\ell}}\hat{S}_{\Delta t}^{n+1-\ell}T_{\Delta t}\int_{t_{\ell-1}}^{\tau}\Big(S(\tau-\rho)-T_{\Delta t}\Big)u_{t_{\ell-1},u^{\ell-1}}(\rho)dW(\rho)dW(\tau)\\
   &-\sum_{\ell=1}^{n+1}\int_{t_{\ell-1}}^{t_{\ell}}\hat{S}_{\Delta t}^{n+1-\ell}T_{\Delta t}\int_{t_{\ell-1}}^{\tau}
   T_{\Delta t}\Big(u_{t_{\ell-1},u^{\ell-1}}(\rho)-u^{\ell-1}\Big)dW(\rho)dW(\tau)\\
   &-\sum_{\ell=1}^{n+1}\hat{S}_{\Delta t}^{n+1-\ell}T_{\Delta t}^2u^{\ell-1}\int_{t_{\ell-1}}^{t_{\ell}}\int_{t_{\ell-1}}^{\tau}dW(\rho)dW(\tau).
 \end{align*}
 We claim that the last term in the above equality has the form
 \begin{align*}
   \sum_{\ell=1}^{n+1}\hat{S}_{\Delta t}^{n+1-\ell}T_{\Delta t}^2u^{\ell-1}\int_{t_{\ell-1}}^{t_{\ell}}\int_{t_{\ell-1}}^{\tau}dW(\rho)dW(\tau)=\frac12\sum_{\ell=1}^{n+1}\hat{S}_{\Delta t}^{n+1-\ell}T_{\Delta t}^2u^{\ell-1}\Big((\Delta W_{\ell-1})^2-F_{\phi}\Delta t\Big).
 \end{align*}
 In fact,
 \begin{align}\label{eq1}
  & \sum_{\ell=1}^{n+1}\hat{S}_{\Delta t}^{n+1-\ell}T_{\Delta t}^2u^{\ell-1}\int_{t_{\ell-1}}^{t_{\ell}}\int_{t_{\ell-1}}^{\tau}dW(\rho)dW(\tau)\\
   &=\sum_{k_1,k_2\in{\mathbb N}}\sum_{\ell=1}^{n+1}\hat{S}_{\Delta t}^{n+1-\ell}T_{\Delta t}^2u^{\ell-1}\phi e_{k_1}\phi e_{k_2}\int_{t_{\ell-1}}^{t_{\ell}}\int_{t_{\ell-1}}^{\tau}d\beta_{k_1}(\rho)d\beta_{k_2}(\tau)\nonumber\\
   &=\sum_{k_1=k_2\in{\mathbb N}}\sum_{\ell=1}^{n+1}\hat{S}_{\Delta t}^{n+1-\ell}T_{\Delta t}^2u^{\ell-1}\phi e_{k_1}\phi e_{k_2}\int_{t_{\ell-1}}^{t_{\ell}}\int_{t_{\ell-1}}^{\tau}d\beta_{k_1}(\rho)d\beta_{k_1}(\tau)\nonumber\\
   &\quad+\sum_{k_1<k_2}\sum_{\ell=1}^{n+1}\hat{S}_{\Delta t}^{n+1-\ell}T_{\Delta t}^2u^{\ell-1}\phi e_{k_1}\phi e_{k_2}\int_{t_{\ell-1}}^{t_{\ell}}\int_{t_{\ell-1}}^{\tau}d\beta_{k_1}(\rho)d\beta_{k_2}(\tau)\nonumber\\
   &\quad+\sum_{k_1>k_2}\sum_{\ell=1}^{n+1}\hat{S}_{\Delta t}^{n+1-\ell}T_{\Delta t}^2u^{\ell-1}\phi e_{k_1}\phi e_{k_2}\int_{t_{\ell-1}}^{t_{\ell}}\int_{t_{\ell-1}}^{\tau}d\beta_{k_1}(\rho)d\beta_{k_2}(\tau)\nonumber\\
&=I+II+III.\nonumber
 \end{align}
 Due to $\int_{t_{\ell-1}}^{t_{\ell}}\int_{t_{\ell-1}}^{\tau}d\beta_{k_1}(\rho)d\beta_{k_1}(\tau)=\frac12\Big((\Delta\beta_{k_1})^2-\Delta t F_{\phi}\Big)$, we have
 $$
 I=\frac12\sum_{k_1=k_2\in{\mathbb N}}\sum_{\ell=1}^{n+1}\hat{S}_{\Delta t}^{n+1-\ell}T_{\Delta t}^2u^{\ell-1}\phi e_{k_1}\phi e_{k_2}\Big((\Delta\beta_{k_1})^2-\Delta t F_{\phi}\Big).
 $$
 We change the index of $k_1$ and $k_2$ in the last term of \eqref{eq1} to obtain
 $$
 III=\sum_{k_2>k_1}\sum_{\ell=1}^{n+1}\hat{S}_{\Delta t}^{n+1-\ell}T_{\Delta t}^2u^{\ell-1}\phi e_{k_2}\phi e_{k_1}\int_{t_{\ell-1}}^{t_{\ell}}\int_{t_{\ell-1}}^{\tau}d\beta_{k_2}(\rho)d\beta_{k_1}(\tau),
 $$
 and
 \begin{align*}
 II+III=&\sum_{k_1<k_2}\sum_{\ell=1}^{n+1}\hat{S}_{\Delta t}^{n+1-\ell}T_{\Delta t}^2u^{\ell-1}\phi e_{k_1}\phi e_{k_2}\left[\int_{t_{\ell-1}}^{t_{\ell}}\int_{t_{\ell-1}}^{\tau}d\beta_{k_1}(\rho)d\beta_{k_2}(\tau)
 +\int_{t_{\ell-1}}^{t_{\ell}}\int_{t_{\ell-1}}^{\tau}d\beta_{k_2}(\rho)d\beta_{k_1}(\tau)\right]\\
 =&\sum_{k_1<k_2}\sum_{\ell=1}^{n+1}\hat{S}_{\Delta t}^{n+1-\ell}T_{\Delta t}^2u^{\ell-1}\phi e_{k_1}\phi e_{k_2}\Delta\beta_{k_1}\Delta\beta_{k_2}.
 \end{align*}
 Combining them together we may prove the claim.

It follows from the rearrangement of the last three lines of \eqref{6} that,
 \begin{align*}
   \mathcal{D}=&-\sum_{\ell=1}^{n+1}\int_{t_{\ell-1}}^{t_{\ell}}\hat{S}_{\Delta t}^{n+1-\ell}T_{\Delta t}\int_{t_{\ell-1}}^{\tau}\Big(S(\tau-\rho)-T_{\Delta t}\Big)u_{t_{\ell-1},u^{\ell-1}}(\rho)dW(\rho)dW(\tau)\\
   &-\sum_{\ell=1}^{n+1}\int_{t_{\ell-1}}^{t_{\ell}}\hat{S}_{\Delta t}^{n+1-\ell}T_{\Delta t}\int_{t_{\ell-1}}^{\tau}
   T_{\Delta t}\Big(u_{t_{\ell-1},u^{\ell-1}}(\rho)-u^{\ell-1}\Big)dW(\rho)dW(\tau)\\
   &-\frac12\sum_{\ell=1}^{n+1}\hat{S}_{\Delta t}^{n+1-\ell}T_{\Delta t}^2 u^{\ell-1}\big((\Delta W_{\ell-1})^2-(\Delta\hat{W}_{\ell-1})^2\big)
   +\frac14\sum_{\ell=1}^{n+1}\hat{S}_{\Delta t}^{n+1-\ell}T_{\Delta t}^2 (u^{\ell}-u^{\ell-1})(\Delta\hat{W}_{\ell-1})^2 \\
   &-\frac12\sum_{\ell=1}^{n+1}\int_{t_{\ell-1}}^{t_{\ell}}S(t_{n+1}-\tau)u_{t_{\ell-1},u^{\ell-1}}(\rho)F_{\phi}d\tau
   +\frac12\sum_{\ell=1}^{n+1}\hat{S}_{\Delta t}^{n+1-\ell}T_{\Delta t}^2u^{\ell-1}F_{\phi}\Delta t\\
   &+i\sum_{\ell=1}^{n+1}\hat{S}_{\Delta t}^{n+1-\ell}T_{\Delta t}u^{\ell-\frac12}\Big(\Delta \tilde{W}_{\ell-1}-\Delta W_{\ell-1}\Big).
 \end{align*}
 The estimates of the first two lines come from It\^o isometry, and are bounded by $K\Delta t^2$. By the properties of the truncated Wiener process, the estimate of the last line is similar to that of $\mathcal{B}^2$, and is bounded also by $K\Delta t^2$.

 Combing all these analysis above, we obtain
 \begin{align*}
   \mathbb{E}\|u(t_{n+1})-u^{n+1}\|_{{\mathbb L}^2}^2\leq K\Delta t^2+K\Delta t\sum_{\ell=1}^{n+1}\|u(t_{\ell-1})-u^{\ell-1}\|_{{\mathbb L}^2}^2.
 \end{align*}
 Therefore, Gronwall's lemma leads to the assertion.
\end{proof}

\subsection{Spatial error}
We state the spatial error estimate of the symplectic local discontinuous Galerkin method (\ref{LDG22}) for the stochastic linear Schr\"odinger equation (\ref{NLS}).
\begin{theorem}\label{the111}
Assume $u_{0}\in  L^2(\Omega; {\mathbb H}^{k+2})$ and $\phi\in\mathcal{L}_{2}({\mathbb L}^2;{\mathbb H}^{k+2})$. Let $u_{h}^{n}$ be the numerical solution of the symplectic local discontinuous Galerkin method \eqref{LDG22}.
Then there exists a constant $h_{0}>0$ such that for $h\leq h_{0}$,
\begin{equation}
\mathbb{E}\|u^{n}-u_{h}^{n}\|_{{\mathbb L}^2}^2\leq Ch^{2k+2}+C\Delta t^{-1}h^{2k+2}.
\end{equation}
\end{theorem}
\begin{proof}
We split the proof into two steps:

\textbf{Step 1: The error equation.}

Notice that the method (\ref{LDG22}) is also satisfied when the numerical solutions $r_{h},p_{h},s_{h},q_{h}$ are replaced by the exact solutions $r,p=s_{x},s,q=s_{x}$. For each fixed $t_{n}$, we can obtain the cell error equation
\begin{equation}\label{errorequation}
\begin{split}
\mathfrak{B}_{j}&(r^{n}-r_{h}^{n},p^{n}-p_{h}^{n},s^{n}-s_{h}^{n},q^{n}-q_{h}^{n};\nu_{h},\omega_{h},\alpha_{h},\beta_{h})\\
&=\int_{I_{j}}[r^{n+1}-r_{h}^{n+1}]\nu_{h}dx-\int_{I_{j}}[r^{n}-r_{h}^{n}]\nu_{h}dx+\Delta t\int_{I_{j}}(p^{n+\frac{1}{2}}-p_{h}^{n+\frac{1}{2}})(\nu_{h})_{x}dx\\
&-\int_{I_{j}}(s^{n+\frac{1}{2}}-s_{h}^{n+\frac{1}{2}})\nu_{h}\Delta \tilde{W}_{n}dx-\Delta t\int_{I_{j}}(p^{n+\frac{1}{2}}-p_{h}^{n+\frac{1}{2}})\omega_{h}dx-\Delta t\int_{I_{j}}(s^{n+\frac{1}{2}}-s_{h}^{n+\frac{1}{2}})(\omega_{h})_{x}dx\\
&-\Delta t\int_{I_{j}}(s^{n+\frac{1}{2}}-s_{h}^{n+\frac{1}{2}})Q_{h}\nu_{h}dx+\Delta t\int_{I_{j}}(r^{n+\frac{1}{2}}-r_{h}^{n+\frac{1}{2}})Q_{h}\alpha_{h}dx\\
&-\Delta t\int_{I_{j}}(q^{n+\frac{1}{2}}-q_{h}^{n+\frac{1}{2}})(\alpha_{h})_{x}dx+\int_{I_{j}}[s^{n+1}-s_{h}^{n+1}]\alpha_{h}dx-\int_{I_{j}}[s^{n}-s_{h}^{n}]\alpha_{h}dx\\
&+\int_{I_{j}}(r^{n+\frac{1}{2}}-r_{h}^{n+\frac{1}{2}})\alpha_{h}\Delta \tilde{W}_{n}dx-\Delta t\int_{I_{j}}(q^{n+\frac{1}{2}}-q_{h}^{n+\frac{1}{2}})\beta_{h}dx-\Delta t\int_{I_{j}}(r^{n+\frac{1}{2}}-r_{h}^{n+\frac{1}{2}})(\beta_{h})_{x}dx\\
&-\Delta t[(p^{n+\frac{1}{2}}-\hat{p}^{n+\frac{1}{2}})\nu_{h}^{-}]_{j+\frac{1}{2}}+\Delta t[(p^{n+\frac{1}{2}}-\hat{p}^{n+\frac{1}{2}})\nu_{h}^{+}]_{j-\frac{1}{2}}+\Delta t[(s^{n+\frac{1}{2}}-\hat{s}^{n+\frac{1}{2}})\omega_{h}^{-}]_{j+\frac{1}{2}}\\
&-\Delta t[(s^{n+\frac{1}{2}}-\hat{s}^{n+\frac{1}{2}})\omega_{h}^{+}]_{j-\frac{1}{2}}+\Delta t[(q^{n+\frac{1}{2}}-\hat{q}^{n+\frac{1}{2}})\alpha_{h}^{-}]_{j+\frac{1}{2}}-\Delta t[(q^{n+\frac{1}{2}}-\hat{q}^{n+\frac{1}{2}})\alpha_{h}^{+}]_{j-\frac{1}{2}}\\
&+\Delta t[(r^{n+\frac{1}{2}}-\hat{r}^{n+\frac{1}{2}})\beta_{h}^{-}]_{j+\frac{1}{2}}-\Delta t[(r^{n+\frac{1}{2}}-\hat{r}^{n+\frac{1}{2}})\beta_{h}^{+}]_{j-\frac{1}{2}}=0
\end{split}
\end{equation}
for all $\nu_{h},\omega_{h},\alpha_{h},\beta_{h}\in V_{h}^{k}$.

Summing over $j$, the error equation becomes
\begin{equation}
\sum_{j=1}^{J}\mathfrak{B}_{j}(r^{n}-r_{h}^{n},p^{n}-p_{h}^{n},s^{n}-s_{h}^{n},q^{n}-q_{h}^{n};\nu_{h},\omega_{h},\alpha_{h},\beta_{h})=0
\end{equation}
for all $\nu_{h},\omega_{h},\alpha_{h},\beta_{h}\in V_{h}^{k}$.

Denoting
\begin{equation}
\begin{split}
&\varepsilon^{n}=\mathcal{P}^{-}r^{n}-r_{h}^{n},~\xi^{n}=\mathcal{P}q^{n}-q_{h}^{n},~\eta^{n}=\mathcal{P}^{-}s^{n}-s_{h}^{n},~\zeta^{n}=p_{h}^{n}-\mathcal{P}p^{n},\\
&\varepsilon_{e}^{n}=\mathcal{P}^{-}r^{n}-r^{n},~\xi_{e}^{n}=\mathcal{P}q^{n}-q^{n},~\eta_{e}^{n}=\mathcal{P}^{-}s^{n}-s^{n},~\zeta_{e}^{n}=p^{n}-\mathcal{P}p^{n},
\end{split}
\end{equation}
where ${\mathcal P}$ is the standard ${\mathbb L}^{2}$-projection of a function $\omega$ with $k+1$ continuous derivatives into space $V_{h}^{k}$, $\mathcal{P}^{-}$ is a special projector into $V_{h}^{k}$, which satisfies, for each $j$,
$$
\int_{I_{j}}(\mathcal{P}^{-}\omega(x)-\omega(x))\nu(x)dx=0,~~\forall \nu\in P^{k-1}(I_{j}),
$$
and $\mathcal{P}^{-}(\omega(x_{j+\frac{1}{2}}^{-}))=\omega(x_{j+\frac{1}{2}})$.
and taking the test functions
\begin{equation*}
\nu_{h}=\varepsilon^{n+\frac{1}{2}},~\omega_{h}=\xi^{n+\frac{1}{2}},~\alpha_{h}=\eta^{n+\frac{1}{2}},~\beta_{h}=\zeta^{n+\frac{1}{2}},
\end{equation*}
we obtain the important energy equality
\begin{equation}\label{energyequation}
\sum_{j=1}^{J}\mathfrak{B}_{j}(\varepsilon^{n}-\varepsilon_{e}^{n},\zeta_{e}^{n}-\zeta^{n},\eta^{n}-\eta_{e}^{n},\xi^{n}-\xi_{e}^{n};\varepsilon^{n+\frac{1}{2}},\xi^{n+\frac{1}{2}},\eta^{n+\frac{1}{2}},\zeta^{n+\frac{1}{2}})=0.
\end{equation}

Now, we shall prove the theorem by analyzing each terms of (\ref{energyequation}).

\textbf{Step 2: Proof of the main result.}

We consider the left-hand side of the energy equation (\ref{energyequation}). Using the linearity of $\mathfrak{B}_{j}$ with respect to its first group of arguments, we get
\begin{equation}\label{3.34}
\begin{split}
&\mathfrak{B}_{j}(\varepsilon^{n}-\varepsilon_{e}^{n},\zeta_{e}^{n}-\zeta^{n},\eta^{n}-\eta_{e}^{n},\xi^{n}-\xi_{e}^{n};\varepsilon^{n+\frac{1}{2}},\xi^{n+\frac{1}{2}},
\eta^{n+\frac{1}{2}},\zeta^{n+\frac{1}{2}})\\
&=\mathfrak{B}_{j}(\varepsilon^{n},-\zeta^{n},\eta^{n},\xi^{n};\varepsilon^{n+\frac{1}{2}},\xi^{n+\frac{1}{2}},\eta^{n+\frac{1}{2}},\zeta^{n+\frac{1}{2}})
-\mathfrak{B}_{j}(\varepsilon_{e}^{n},-\zeta_{e}^{n},\eta_{e}^{n},\xi_{e}^{n};\varepsilon^{n+\frac{1}{2}},\xi^{n+\frac{1}{2}},\eta^{n+\frac{1}{2}},\zeta^{n+\frac{1}{2}}).
\end{split}
\end{equation}

First, we consider the first term of the right-hand side in (\ref{3.34}), which yields
\begin{equation}\label{3.355}
\begin{split}
\mathfrak{B}_{j}(\varepsilon^{n},&-\zeta^{n},\eta^{n},\xi^{n};\varepsilon^{n+\frac{1}{2}},\xi^{n+\frac{1}{2}},\eta^{n+\frac{1}{2}},\zeta^{n+\frac{1}{2}})\\
&=\frac{1}{2}\int_{I_{j}}\Big((\varepsilon^{n+1})^{2}-(\varepsilon^{n})^{2}\Big)dx+\frac{1}{2}\int_{I_{j}}\Big((\eta^{n+1})^{2}-(\eta^{n})^{2}\Big)dx\\
&+\Delta t[(\zeta^{+}\varepsilon^{-})_{j+\frac{1}{2}}^{n+\frac{1}{2}}-(\zeta^{+}\varepsilon^{+})_{j-\frac{1}{2}}^{n+\frac{1}{2}}]+\Delta t[(\eta^{-}\xi^{-})_{j+\frac{1}{2}}^{n+\frac{1}{2}}-(\eta^{-}\xi^{+})_{j-\frac{1}{2}}^{n+\frac{1}{2}}]\\
&+\Delta t[(\xi^{+}\eta^{-})_{j+\frac{1}{2}}^{n+\frac{1}{2}}-(\xi^{+}\eta^{+})_{j-\frac{1}{2}}^{n+\frac{1}{2}}]+\Delta t[(\varepsilon^{-}\zeta^{-})_{j+\frac{1}{2}}^{n+\frac{1}{2}}-(\varepsilon^{-}\zeta^{+})_{j-\frac{1}{2}}^{n+\frac{1}{2}}]\\
&-\Delta t\underbrace{\int_{I_{j}}[(\eta\xi)_{x}^{n+\frac{1}{2}}+(\varepsilon\zeta)_{x}^{n+\frac{1}{2}}]dx}_{R}.
\end{split}
\end{equation}
From the integration by parts, we arrive at
\begin{equation}\label{3.3100}
\begin{split}
R=\Big[(\eta^{-}\xi^{-})_{j+\frac{1}{2}}^{n+\frac{1}{2}}-(\eta^{+}\xi^{+})_{j-\frac{1}{2}}^{n+\frac{1}{2}}\Big]+\Big[(\varepsilon^{-}\zeta^{-})_{j+\frac{1}{2}}^{n+\frac{1}{2}}-(\varepsilon^{+}\zeta^{+})_{j-\frac{1}{2}}^{n+\frac{1}{2}}\Big].
\end{split}
\end{equation}

Substituting (\ref{3.3100}) into (\ref{3.355}), we have
 \begin{equation}\label{3.35}
\begin{split}
\mathfrak{B}_{j}(\varepsilon^{n},&-\zeta^{n},\eta^{n},\xi^{n};\varepsilon^{n+\frac{1}{2}},\xi^{n+\frac{1}{2}},\eta^{n+\frac{1}{2}},\zeta^{n+\frac{1}{2}})\\
&=\frac{1}{2}\int_{I_{j}}\Big((\varepsilon^{n+1})^{2}-(\varepsilon^{n})^{2}\Big)dx+\frac{1}{2}\int_{I_{j}}\Big((\eta^{n+1})^{2}-(\eta^{n})^{2}\Big)dx+\Delta t[\hat{\Phi}_{j+\frac{1}{2}}^{n+\frac{1}{2}}-\hat{\Phi}_{j-\frac{1}{2}}^{n+\frac{1}{2}}],
\end{split}
\end{equation}
where $\hat{\Phi}=\xi^{+}\eta^{-}+\zeta^{+}\varepsilon^{-}$.

As for the second term of the right-hand side in (\ref{3.34}) , we have
\begin{equation}
\mathfrak{B}_{j}(\varepsilon_{e}^{n},-\zeta_{e}^{n},\eta_{e}^{n},\xi_{e}^{n};\varepsilon^{n+\frac{1}{2}},\xi^{n+\frac{1}{2}},\eta^{n+\frac{1}{2}},\zeta^{n+\frac{1}{2}})=I+II+III+IV+V,
\end{equation}
where
\begin{equation*}
\begin{split}
&I=\int_{I_{j}}(\varepsilon_{e}^{n+1}-\varepsilon_{e}^{n})\varepsilon^{n+\frac{1}{2}}dx+\int_{I_{j}}(\eta_{e}^{n+1}-\eta_{e}^{n})\eta^{n+\frac{1}{2}}dx,\\
&II=\Delta t\int_{I_{j}}\Big((\zeta_{e}\xi)^{n+\frac{1}{2}}-(\zeta_{e}\varepsilon_{x})^{n+\frac{1}{2}}-(\eta_{e}\xi_{x})^{n+\frac{1}{2}}-(\xi_{e}\eta_{x})^{n+\frac{1}{2}}\\
&~~~~~~~~~-(\xi_{e}\zeta)^{n+\frac{1}{2}}-(\varepsilon_{e}\zeta_{x})^{n+\frac{1}{2}}- Q_{h}(\eta_{e}\varepsilon)^{n+\frac{1}{2}}+ Q_{h}(\varepsilon_{e}\eta)^{n+\frac{1}{2}}\Big)dx,
\end{split}
\end{equation*}
\begin{equation*}
\begin{split}
&III=\int_{I_{j}}\varepsilon_{e}^{n+\frac{1}{2}}\eta^{n+\frac{1}{2}}\Delta \tilde{W}_{n}dx,\quad IV=-\int_{I_{j}}\eta_{e}^{n+\frac{1}{2}}\varepsilon^{n+\frac{1}{2}}\Delta \tilde{W}_{n}dx,\\
&V=\Delta t\Big[(\zeta_{e}^{+}\varepsilon^{-})_{j+\frac{1}{2}}^{n+\frac{1}{2}}-(\zeta_{e}^{+}\varepsilon^{+})_{j-\frac{1}{2}}^{n+\frac{1}{2}}-(\eta_{e}^{-}\xi^{-})_{j+\frac{1}{2}}^{n+\frac{1}{2}}+(\eta_{e}^{-}\xi^{+})_{j-\frac{1}{2}}^{n+\frac{1}{2}}\\
&~~~~~~~~~-(\xi_{e}^{e}\eta^{-})_{j+\frac{1}{2}}^{n+\frac{1}{2}}-(\xi_{e}^{+}\eta^{+})_{j-\frac{1}{2}}^{n+\frac{1}{2}}-(\varepsilon_{e}^{-}\zeta^{-})_{j+\frac{1}{2}}^{n+\frac{1}{2}}+(\varepsilon_{e}^{-}\zeta^{+})_{j-\frac{1}{2}}^{n+\frac{1}{2}}\Big].
\end{split}
\end{equation*}

By using the simple inequality $ab\leq \frac{a^{2}}{4}+b^{2}$, and the standard approximation theory (\ref{Projection}) on $\varepsilon^{e}$, and $\eta^{e}$, we have
\begin{align*}
I &\leq \|\varepsilon_{e}^{n+1}-\varepsilon_{e}^{n}\|_{{\mathbb L}^2(I_j)}\|\varepsilon^{n+\frac12}\|_{{\mathbb L}^2(I_j)}+
          \|\eta_{e}^{n+1}-\eta_{e}^{n}\|_{{\mathbb L}^2(I_j)}\|\eta^{n+\frac12}\|_{{\mathbb L}^2(I_j)}\\
          &\leq C\Delta t^{-1}\|\varepsilon_{e}^{n+1}-\varepsilon_{e}^{n}\|_{{\mathbb L}^2(I_j)}^2+C\Delta t \|\varepsilon^{n+\frac12}\|_{{\mathbb L}^2(I_j)}^2
        +C\Delta t^{-1}\|\eta_{e}^{n+1}-\eta_{e}^{n}\|_{{\mathbb L}^2(I_j)}^2+C\Delta t\|\eta^{n+\frac12}\|_{{\mathbb L}^2(I_j)}^2,
\end{align*}
where $\|\varepsilon_{e}^{n+1}-\varepsilon_{e}^{n}\|_{{\mathbb L}^2(I_j)}=\|\mathcal{P}^{-}(r^{n+1}-r^{n})-(r^{n+1}-r^{n})\|_{{\mathbb L}^2(I_j)}$ and
 $\|\eta_{e}^{n+1}-\eta_{e}^{n}\|_{{\mathbb L}^2(I_j)}=\|\mathcal{P}^{-}(s^{n+1}-s^{n})-(s^{n+1}-s^{n})\|_{{\mathbb L}^2(I_j)}$.
 It is well know that for any $\omega\in {\mathbb H}^{k+1}(\mathbb{R})$
\begin{equation}\label{Projection}
\|\breve{\omega}(x)\|_{{\mathbb L}^{2}}+h\|\breve{\omega}(x)\|_{{\mathbb L}^{\infty}}+\sqrt{h}\|\breve{\omega}(x)\|_{\Gamma_{h}}\leq C\|\omega\|_{{\mathbb H}^{k+1}}h^{k+1}
\end{equation}
where $\breve{\omega}=\mathcal{P}\omega-\omega$ or $\breve{\omega}=\mathcal{P}^{-}\omega-\omega$. The positive constant $C$ is independent of $h$, and $\Gamma_{h}$ is the usual $L^{2}$-norm on the cell interfaces of the mesh, which for this one-dimensional case is
$
\|\nu\|_{\Gamma_{h}}^{2}=\sum_{j=1}^{J}\Big((\nu^{-}_{j+\frac{1}{2}})^{2}+(\nu^{+}_{j-\frac{1}{2}})^{2}\Big).
$

  Summing over $j$ and taking
 expectation, utilizing the property of projection and the estimate of $\mathbb{E}\|r^{n+1}-r^{n}\|_{{\mathbb H}^{k+1}}^{2}$ (see Lemma \ref{th66} with $p=1$) and Lemma \ref{th4}, we have
 \begin{equation}
   \mathbb{E}(\sum_{j=1}^{J}I)\leq C\mathbb{E}\|u_0\|_{{\mathbb H}^{k+2}}^2 h^{2k+2}+C\Delta t \mathbb{E}\|\varepsilon^{n+\frac12}\|_{{\mathbb L}^2([L_f,L_r])}^2+C\Delta t \mathbb{E}\|\eta^{n+\frac12}\|_{{\mathbb L}^2([L_f,L_r])}^2.
 \end{equation}

From the property of the projections $\mathcal{P}$ and $\mathcal{P}^{-}$, it follows that all the terms in $II$ except the last two terms are actually zero. We can get the estimates for $II$ via Young's inequality and Lemma \ref{th4},
\begin{align*}
\mathbb{E}(\sum_{j=1}^{J}II)&\leq C\mathbb{E}\big(\|r^n\|_{{\mathbb H}^{k+2}}^2+\|s^n\|_{{\mathbb H}^{k+2}}^2\big)\Delta th^{2k+2}+\frac{\Delta t}{4}\mathbb{E}\|\varepsilon^{n+\frac12}\|_{{\mathbb L}^2([L_f,L_r])}^2+\frac{\Delta t}{4} \mathbb{E}\|\eta^{n+\frac12}\|_{{\mathbb L}^2([L_f,L_r])}^2\\
&\leq C\mathbb{E}\|u_0\|_{{\mathbb H}^{k+2}}^2\Delta th^{2k+2}+\frac{\Delta t}{4}\mathbb{E}\|\varepsilon^{n+\frac12}\|_{{\mathbb L}^2([L_f,L_r])}^2+\frac{\Delta t}{4} \mathbb{E}\|\eta^{n+\frac12}\|_{{\mathbb L}^2([L_f,L_r])}^2
\end{align*}

For the third term $III$, we have
\begin{equation*}
\begin{split}
\mathbb{E}(\sum_{j=1}^{J}III)&=\frac14 \mathbb{E}\int_{L_f}^{L_r}(\varepsilon_{e}^{n+1}-\varepsilon_{e}^{n})(\eta^{n+1}-\eta^{n})\Delta \tilde{W}_{n}dx\\
&+\frac12 \mathbb{E}\int_{L_f}^{L_r}(\varepsilon_{e}^{n+1}-\varepsilon_{e}^{n})\eta^{n}\Delta \tilde{W}_{n}dx
+\frac12 \mathbb{E}\int_{L_f}^{L_r}\varepsilon_{e}^{n}(\eta^{n+1}-\eta^{n})\Delta \tilde{W}_{n}dx\\
&=:III^a+III^b+III^c.
\end{split}
\end{equation*}
For term $III^a$, using Young's inequality, Lemma \ref{th4} and Lemma \ref{th66} with $p=2$, we have
\begin{align*}
  III^a&\leq \frac14\mathbb{E}\Big( \|\varepsilon_{e}^{n+1}-\varepsilon_{e}^{n}\|_{{\mathbb L}^2([L_f,L_r])}\|\eta^{n+1}-\eta^{n}\|_{{\mathbb L}^2([L_f,L_r])}\|\Delta \tilde{W}_{n}\|_{{\mathbb L}^{\infty}([L_f,L_r])}\Big)\\
  &\leq C\Delta t\mathbb{E} \|\eta^{n+1}-\eta^{n}\|_{{\mathbb L}^2([L_f,L_r])}^2+C\Delta t^{-1}\mathbb{E}\Big(\|\Delta \tilde{W}_{n}\|_{{\mathbb L}^{\infty}([L_f,L_r])}^2\|\varepsilon_{e}^{n+1}-\varepsilon_{e}^{n}\|_{{\mathbb L}^2([L_f,L_r])}^2\Big)\\
  &\leq C\Delta t\mathbb{E} \|\eta^{n+1}-\eta^{n}\|_{{\mathbb L}^2([L_f,L_r])}^2
+C\Delta t^{-1}\Big(h^{2k+2}\mathbb{E}\|\Delta \tilde{W}_{n}\|_{{\mathbb L}^{\infty}([L_f,L_r])}^4+h^{-(2k+2)}\mathbb{E}\|\varepsilon_{e}^{n+1}-\varepsilon_{e}^{n}\|_{{\mathbb L}^2([L_f,L_r])}^4\Big)\\
  &\leq C\Delta t\mathbb{E} \|\eta^{n+1}\|_{{\mathbb L}^2([L_f,L_r])}^2+C\Delta t\mathbb{E} \|\eta^{n}\|_{{\mathbb L}^2([L_f,L_r])}^2+C\Delta t h^{2k+2}.
\end{align*}
Similarly, for term $III^b$,
\begin{align*}
  III^b &\leq\frac12 \mathbb{E}\Big(\|\varepsilon_{e}^{n+1}-\varepsilon_{e}^{n}\|_{{\mathbb L}^2([L_f,L_r])}\|\eta^{n}\|_{{\mathbb L}^2([L_f,L_r])}\|\Delta \tilde{W}_{n}\|_{{\mathbb L}^{\infty}([L_f,L_r])}\Big)\\
  &\leq C\mathbb{E}\|\varepsilon_{e}^{n+1}-\varepsilon_{e}^{n}\|_{{\mathbb L}^2([L_f,L_r])}^2+C\mathbb{E}\Big(\|\eta^{n}\|_{{\mathbb L}^2([L_f,L_r])}^2\|\Delta \tilde{W}_{n}\|_{{\mathbb L}^{\infty}([L_f,L_r])}^2\Big)\\
  &\leq C\Delta t\mathbb{E} \|\eta^{n}\|_{{\mathbb L}^2([L_f,L_r])}^2+C\Delta th^{2k+2},
\end{align*}
and  for term $III^c$,
\begin{align*}
  III^c &\leq \frac12 \mathbb{E}\Big(\|\varepsilon_{e}^{n}\|_{{\mathbb L}^2([L_f,L_r])}\|\eta^{n+1}-\eta^{n}\|_{{\mathbb L}^2([L_f,L_r])}\|\Delta \tilde{W}_{n}\|_{{\mathbb L}^{\infty}([L_f,L_r])}\Big)\\
  &\leq C\Delta t\mathbb{E}\|\eta^{n+1}-\eta^{n}\|_{{\mathbb L}^2([L_f,L_r])}^2+C\Delta t^{-1}\mathbb{E}\Big(\|\varepsilon_{e}^{n}\|_{{\mathbb L}^2([L_f,L_r])}^2\|\Delta \tilde{W}_{n}\|_{{\mathbb L}^{\infty}([L_f,L_r])}^2\Big)\\
  &\leq C\Delta t\mathbb{E} \|\eta^{n+1}\|_{{\mathbb L}^2([L_f,L_r])}^2+C\Delta t\mathbb{E} \|\eta^{n}\|_{{\mathbb L}^2([L_f,L_r])}^2+C h^{2k+2},
\end{align*}
where in the last inequalities for the estimate of $III^b$ and $III^c$, we use the independent property of Wiener process.
The estimate of term $IV$ is similar as that of term $III$, so we omit the process here.

Finally, $V$ only contains flux difference terms which all vanish upon a summation in $j$.
Combining these together, we know that
\begin{equation*}
\begin{split}
 &\frac{1}{2}\mathbb{E}\Big( \|\varepsilon^{n+1}\|_{{\mathbb L}^2([L_f,L_r])}^2+\|\eta^{n+1}\|_{{\mathbb L}^2([L_f,L_r])}^2\Big)
 -\frac{1}{2}\mathbb{E}\Big(\|\varepsilon^{n}\|_{{\mathbb L}^2([L_f,L_r])}^2+\|\eta^{n}\|_{{\mathbb L}^2([L_f,L_r])}^2\Big)\\
 &\leq C\Delta t \mathbb{E}\|\varepsilon^{n+1}\|_{{\mathbb L}^2([L_f,L_r])}^2+C\Delta t \mathbb{E}\|\varepsilon^{n}\|_{{\mathbb L}^2([L_f,L_r])}^2\\
 &~~~~+C\Delta t \mathbb{E}\|\eta^{n+1}\|_{{\mathbb L}^2([L_f,L_r])}^2+C\Delta t \mathbb{E}\|\eta^{n}\|_{{\mathbb L}^2([L_f,L_r])}^2
 +Ch^{2k+2}+C\Delta th^{2k+2}.
 \end{split}
\end{equation*}
By Gronwall's inequality, there exists a constant $h_0>0$, for $h\leq h_0$, we obtain
\[\mathbb{E}\Big(\|\varepsilon^{n}\|_{{\mathbb L}^2([L_f,L_r])}^2+\|\eta^{n}\|_{{\mathbb L}^2([L_f,L_r])}^2\Big)\leq Ch^{2k+2}+C\Delta t^{-1}h^{2k+2}, \quad \forall n.\]
I.e.,
\begin{equation}\label{spartial}
\mathbb{E}\|u^n-u_{h}^{n}\|_{{\mathbb L}^2}^2\leq Ch^{2k+2}+C\Delta t^{-1}h^{2k+2}.
\end{equation}

The proof is finished.
\end{proof}

\subsection{Main result}
Combining Theorem \ref{Time} and Theorem \ref{the111}, we obtain the error estimate of  (\ref{LDG22}).
\begin{theorem}
Let $u(x,t)$ be the exact solution of the problem \eqref{NLS}, and assume the initial value $u_{0}(x)\in L^2(\Omega;{\mathbb H}^{k+2})$ and $\phi\in{\mathcal L}_2({\mathbb L}^2;{\mathbb H}^{k+2})$ $(k\geq 1)$ . Let $u_{h}^{n}$ be the numerical solution of the symplectic local discontinuous Galerkin method \eqref{LDG22}. Then there exists a constant $h_0>0$ such that for $h\leq h_0$ holds
\begin{equation}
\mathbb{E}\|u(t_{n})-u_{h}^{n}\|_{{\mathbb L}^2}^2\leq C\Delta t^{2}+Ch^{2k+2}+C\Delta t^{-1}h^{2k+2}.
\end{equation}
\end{theorem}
The overall convergence rate is usually expressed in terms of the computational cost of the scheme \citep{RefJentzen}.
Here the computational cost of method \eqref{LDG22} is denoted by $M=N\cdot J$ with $N$ and $J$ being the total grid number in temporal and spacial directions, respectively. In view of the above error bound, it is optimal to choose $N=M^{\frac{2k+2}{2k+5}}$ and $J=M^{\frac{3}{2k+5}}$, i.e., $\Delta t=O(\frac{1}{N})=O\left(\left(\frac{1}{M}\right)^{\frac{2k+2}{2k+5}}\right)$ and $h=O(\frac{1}{J})=O\left(\left(\frac{1}{M}\right)^{\frac{3}{2k+5}}\right)$, and we have the optimal error bound
$$
\Big(\mathbb{E}\|u(t_{n})-u_{h}^{n}\|_{{\mathbb L}^2}^2\Big)^{\frac12}\leq C\left(\frac{1}{M}\right)^{\frac{2k+2}{2k+5}}.
$$

\begin{remark}
  If $k=1$, i.e., the initial data $u_0\in L^2(\Omega;{\mathbb H}^{3})$ and $\phi\in{\mathcal L}_2({\mathbb L}^2;{\mathbb H}^{3})$, then
   the mean-square convergence rate of the method \eqref{LDG22} with respect to the computational cost is $\frac{4}{7}$.
\end{remark}
\begin{remark}
In the section 3, the mean-square convergence was derived for the symplectic local discontinuous Galerkin method \eqref{LDG22} discretized equation \eqref{NLS}. Note that \eqref{NLS} is the linear Schr\"odinger equation.
As for nonlinear equation, truncation strategy may be needed to deal with the nonlinear term, as in \citep{RefBouard1',RefBouard4,RefLiu}. However, things are a bit technical for the error estimation of the symplectic local discontinuous Galerkin method, since if we employ truncated strategy, then it has to start by taking ${\mathbb H}^{\gamma}$-norm ($\gamma>\frac{d}{2}$) on the error equation; see Remark 3.2 in \citep{RefLiu}. It looks like other technical strategy is needed to derive the mean-square convergence for symplectic local discontinuous Galerkin method applied to nonlinear case, and it will be our future work.
\end{remark}

\clearpage

\appendix

\section*{Appendix A.\ Proof of lemma \ref{th4}}
\begin{proof}
We present the proof for $p=1$ in the following, and the general case follows similarly.
  First of all, we rewrite temporal semi-discretization system \eqref{time} into the function of $u^{n}$:
  \begin{align}\label{un}
    u^{n+1}=\hat{S}_{\Delta t}u^{n}-i\Delta t T_{\Delta t}Qu^{n+\frac12}-iT_{\Delta t}u^{n+\frac12}\Delta \tilde{W}_{n},
  \end{align}
  where $u^{n}$ denotes the complex function $r^{n}+is^{n}$, operators are defined by $\hat{S}_{\Delta t}=(I+i\frac{\Delta t}{2}\partial_{xx})^{-1}(I-i\frac{\Delta t}{2}\partial_{xx})$ and
  $T_{\Delta t}=(I+i\frac{\Delta t}{2}\partial_{xx})^{-1}$, where $I$ is an identity operator.

  It is easy to check that the operator $\hat{S}_{\Delta t}$ is isometry in ${\mathbb L}^2$, i.e., $\|\hat{S}_{\Delta t}\|_{{\mathcal L}({\mathbb L}^2;{\mathbb L}^2)}=1$.
  Furthermore, we know that $\|T_{\Delta t}\|_{{\mathcal L}({\mathbb L}^2;{\mathbb L}^2)}\leq 1$. See reference \citep{RefBouard4} for example.

  Next, we replace the function of $u^{n}$ into equation \eqref{un} iteratively. We obtain
  \begin{align}\label{iterate}
    u^{n}=\hat{S}_{\Delta t}^{n}u^{0}-i\Delta t\sum_{\ell=1}^{n}\hat{S}_{\Delta t}^{n-\ell}T_{\Delta t}Qu^{\ell-\frac12}
    -i\sum_{\ell=1}^{n}\hat{S}_{\Delta t}^{n-\ell}T_{\Delta t}u^{\ell-\frac12}\Delta \tilde{W}_{\ell-1}.
  \end{align}
  In order to bound function $u^{n}$, we insert the equality $u^{\ell-\frac12}=\frac12(\hat{S}_{\Delta t}+I)u^{\ell-1}+\frac12\big(u^{\ell}-\hat{S}_{\Delta t}u^{\ell-1}\big)$ into
  the stochastic term and take ${\mathbb H}^{\gamma}$-norm to get
  \begin{equation}\label{unu0}
  \begin{split}
    \|u^{n}\|_{{\mathbb H}^{\gamma}}^{2}\leq & K\|u^0\|_{{\mathbb H}^{\gamma}}^{2}+K\Delta t\sum_{\ell=1}^{n}\|u^{\ell-\frac12}\|_{{\mathbb H}^{\gamma}}^{2}
    +K\Big\|\frac{i}{2}\sum_{\ell=1}^{n}\hat{S}_{\Delta t}^{n-\ell}T_{\Delta t}\big(\hat{S}_{\Delta t}+I\big)u^{\ell-1}\Delta \tilde{W}_{\ell-1}\Big\|_{{\mathbb H}^{\gamma}}^{2}\\
    &+Kn\sum_{\ell=1}^{n}\big\|(u^{\ell}-\hat{S}_{\Delta t}u^{\ell-1})\Delta \tilde{W}_{\ell-1}\big\|_{{\mathbb H}^{\gamma}}^{2}.
  \end{split}
  \end{equation}
  For the third term on the right-hand side of \eqref{unu0}, using the fact that $u^{\ell-1}$ is independent of increment $\Delta\tilde{W}_{\ell-1}$, we have
  \begin{equation}\label{unu2}
  \begin{split}
   {\mathbb E} \Big\|\frac{i}{2}\sum_{\ell=1}^{n}\hat{S}_{\Delta t}^{n-\ell}T_{\Delta t}\big(\hat{S}_{\Delta t}+I\big)u^{\ell-1}\Delta \tilde{W}_{\ell-1}\Big\|_{{\mathbb H}^{\gamma}}^{2}
   &=\sum_{\ell=1}^{n}{\mathbb E}\Big\|\hat{S}_{\Delta t}^{n-\ell}T_{\Delta t}\big(\hat{S}_{\Delta t}+I\big)u^{\ell-1}\Delta \tilde{W}_{\ell-1}\Big\|_{{\mathbb H}^{\gamma}}^{2}\\
   &\leq K\Delta t\sum_{\ell=1}^{n}{\mathbb E}\|u^{\ell-1}\|_{{\mathbb H}^{\gamma}}^{2}.
   \end{split}
  \end{equation}
  To estimate the last term on the right-hand side of \eqref{unu0}, we note that
  \begin{equation}\label{unu4}
  \begin{split}
    (u^{\ell}-\hat{S}_{\Delta t}u^{\ell-1})\Delta \tilde{W}_{\ell-1}=&-i\Delta tT_{\Delta t}Qu^{\ell-\frac12}\Delta \tilde{W}_{\ell-1}-\frac{i}{2}T_{\Delta t}(\hat{S}_{\Delta t}+I)u^{\ell-1}(\Delta\tilde{W}_{\ell-1})^2\\
    &-\frac{i}{2}T_{\Delta t}\Big((u^{\ell}-\hat{S}_{\Delta t}u^{\ell-1})\Delta \tilde{W}_{\ell-1}\Big)\Delta\tilde{W}_{\ell-1}.
  \end{split}
  \end{equation}
  Taking $L^2(\Omega;{\mathbb H}^{\gamma})$-norm to obtain
  \begin{equation}\label{unu1}
  \begin{split}
  {\mathbb E}\big\|(u^{\ell}-\hat{S}_{\Delta t}u^{\ell-1})\Delta \tilde{W}_{\ell-1}\big\|_{{\mathbb H}^{\gamma}}^2
  \leq& K\Delta t^2 (\Delta t\kappa^2){\mathbb E}\|u^{\ell-\frac12}\|_{{\mathbb H}^{\gamma}}^2+K\Delta t^2{\mathbb E}\|u^{\ell-1}\|_{{\mathbb H}^{\gamma}}^2\\
  &+K(\Delta t\kappa^2){\mathbb E}\|(u^{\ell}-\hat{S}_{\Delta t}u^{\ell-1})\Delta \tilde{W}_{\ell-1}\|_{{\mathbb H}^{\gamma}}^2,
  \end{split}
  \end{equation}
  where we use the embedding ${\mathbb H}^1\hookrightarrow {\mathbb L}^{\infty}$ for $\gamma=0$, or use $\|fg\|_{{\mathbb H}^{\gamma}}\leq K\|f\|_{{\mathbb H}^{\gamma}}\|g\|_{{\mathbb H}^{\gamma}}$ for $\gamma\geq 1$.
  Note that there exists a constant $\Delta t^{*}>0$ such that $K(\Delta t\kappa^2)\leq \frac12<1$ for $\Delta t\leq \Delta t^{*}$ (here $K$ is the same as the last term on the right-hand side of \eqref{unu1}), which leads to
  \begin{equation}\label{unu3}
    \frac12{\mathbb E}\big\|(u^{\ell}-\hat{S}_{\Delta t}u^{\ell-1})\Delta \tilde{W}_{\ell-1}\big\|_{{\mathbb H}^{\gamma}}^2
    \leq K\Delta t^2 \Big({\mathbb E}\|u^{\ell}\|_{{\mathbb H}^{\gamma}}^2+{\mathbb E}\|u^{\ell-1}\|_{{\mathbb H}^{\gamma}}^2\Big).
  \end{equation}
  Combining inequalities \eqref{unu0}, \eqref{unu2} and \eqref{unu3} together, we have
 \begin{align*}
    {\mathbb E}\|u^{n}\|_{{\mathbb H}^{\gamma}}^{2}\leq K+K\Delta t\sum_{\ell=0}^{n}{\mathbb E}\|u^{\ell}\|_{{\mathbb H}^{\gamma}}^{2},
  \end{align*}
 where the positive constant $K$ depends on $p$, $T$, the $L^2$-norm of operator $T_{\Delta t}$, $\|u^{0}\|_{H^{\gamma}}$, but not depends on $\Delta t$.
  The discrete Gronwall's lemma leads to the assertion.
  \end{proof}
  \section*{Appendix B.\ Proof of lemma \ref{th66}}
  \begin{proof}
    We present the proof for $p=1$. The estimation is similar as the proof of the last term on the right-hand side of \eqref{unu0}; see estimations \eqref{unu4}-\eqref{unu3}.  Start from equation \eqref{un},
  \begin{align*}
    u^{n+1}-u^{n}=(\hat{S}_{\Delta t}-I)u^{n}-i\Delta t T_{\Delta t}Qu^{n+\frac12}-iT_{\Delta t}u^{n+\frac12}\Delta \tilde{W}_{n}.
  \end{align*}

  Since $\|\hat{S}_{\Delta t}-I\|_{\mathcal{L}({\mathbb H}^{\gamma},{\mathbb H}^{\gamma-1})}\leq K\Delta t^{\frac12}$, we take $L^{2}(\Omega;{\mathbb H}^{\gamma-1})$-norm on both sides of the above equation and get
  \begin{equation}\label{unu5}
  \begin{split}
    \mathbb{E}\|u^{n+1}-u^{n}\|_{{\mathbb H}^{\gamma-1}}^{2}\leq&
    K\Delta t \mathbb{E}\|u^{n}\|_{{\mathbb H}^{\gamma}}^{2}+K\Delta t^{2}\mathbb{E}\Big(\|u^{n}\|_{{\mathbb H}^{\gamma-1}}^{2}+\|u^{n+1}\|_{{\mathbb H}^{\gamma-1}}^{2}\Big)\\
    &+K\Delta t{\mathbb E}\|u^{n}\|_{{\mathbb H}^{\gamma-1}}^{2}+K(\Delta t\kappa^2){\mathbb E}\|u^{n+1}-u^{n}\|_{{\mathbb H}^{\gamma-1}}^2.
    \end{split}
  \end{equation}
  there exists a constant $\Delta t^{*}>0$ such that $K(\Delta t\kappa^2)\leq \frac12<1$ for $\Delta t\leq \Delta t^{*}$ (here $K$ is the same as the last term on the right-hand side of \eqref{unu5}), which leads to
  \begin{equation*}
    \frac12 \mathbb{E}\|u^{n+1}-u^{n}\|_{{\mathbb H}^{\gamma-1}}^{2}\leq K\Delta t{\mathbb E}\|u^{n}\|_{{\mathbb H}^{\gamma}}^{2}\leq K\Delta t.
  \end{equation*}
   This completes the proof.
  \end{proof}

\end{document}